\documentclass{article}
\usepackage[utf8]{inputenc}
\usepackage{amsmath, amsfonts, amsthm, amssymb}
\usepackage{graphicx}
\usepackage[all]{xy}
\usepackage{hyperref}
\newcommand{\cuerpo}[1]{\mathbb{#1}}
\newcommand{\Z}{\cuerpo{Z}}
\newcommand{\R}{\cuerpo{R}}

\newtheorem{Teo}{Theorem}[section]
\newtheorem{Pro}{Proposition}[section]
\newtheorem{Afir}{Affirmation}[section]
\newtheorem{Obs}{Observation}[section]
\newtheorem{Cor}[Teo]{Corolary}
\newtheorem{Exp}[Teo]{Example}
\newtheorem{Def}[Teo]{Definition}

\begin{document}
	
	\title{ Categorical Equivalence between $PMV_f$- product algebras and semi-low $f_u$-rings. }
	\author{Lilian J. Cruz \thanks{Universidad del Valle.} \\ lilian.cruz@correounivalle.edu.co \and Yuri A. Poveda \thanks{Universidad Tecnológica de Pereira.} \\ yapoveda@utp.edu.co}

	\maketitle              
	
	\begin{abstract}
		An explicit categorical equivalence is defined between a proper subvariety of the class of $PMV$-algebras, as  defined by Di Nola and Dvure$\check{c}$enskij, to be called $PMV_f$-algebras, and the category of semi-low $f_u$-rings. This categorical representation is done using the prime spectrum of the $MV$-algebras, through the equivalence between $MV$-algebras and $l_u$-groups established by Mundici, from the perspective of the Dubuc-Poveda approach, that extends the construction defined by Chang on chains. As a particular case, semi-low $f_u$-rings associated to Boolean algebras are characterized. Besides we show that class of $PMV_f$-algebras is coextensive.\\ 
		
	\end{abstract}

	\section{Introduction}
	In this paper the categorical equivalence is described,  between a classical universal algebra variety, subvariety of  the class of  $PMV$-algebras, the $PMV_f$-algebras  and the category of semi-low $f_u$-rings. This intermediate variety is a proper subvariety of the  $PMV$-algebras defined by Di Nola y Dvure$\check{c}$enskij \cite{DiNola2}. On the other hand, the variety of commutative unitary $PMV$-algebras studied by Montagna \cite{Montagna}, to be called in this paper $PMV_1$-algebras, is a proper subvariety of the  $PMV_f$. Estrada \cite{Estrada} defined the variety of $MVW$-rigs, and we defined the variety of $PMV_f$ from it. The $MVW$-rigs contains strictly the variety of $PMV$-algebras. Every $MV$-algebra, with the infimum as product, is an $MVW$-rig (Proposition \ref{proinfimo}), it can happen that it is not a $PMV$-algebra; for example, the $\L$ukasiewicz $MV$-algebras or the $MV$-algebra $[0,1]$. 
	
	The equivalence between the category of $PMV_f$-algebras  and the category of semi-low $f_u$-rings is established based of the equivalence proved by  Mundici \cite{Mundici3}, but applying the construction introduced by Dubuc-Poveda \cite{Yuri3}, since it does not require the \textit{good sequences}, and relies in the representation of any $MV$-algebra as a subdirect product of totally ordered $MV$-algebras, that will be called from here on chain $MV$-algebras or $MV$-chain. This representation only requires the prime spectrum of an $MV$-algebras and the equivalence between chain $MV$-algebras and the chain $l_u$-groups, established by Chang \cite{Chang2}.
	
	It is  proved that for the representation established in this paper, it is enough with the prime spectrum of the subjacent $MV$-algebra, since every $PMV_f$-algebra $A$ is a $PMV$-algebra that satisfies that $xy \leq x \wedge y$, y $x(y \ominus z)= xy \ominus xz$, for every  $x,y, z \in A$, and every prime ideal of the subjacent $MV$-algebra is an ideal of the $PMV_f$-algebra $A$.
	
	This construction finds explicit representations for the rings associated to notable examples of $PMV_f$-algebras. For example, the  $MV$-algebra $[0,1]$ with the usual product, the $MV$-algebra of the functions from $[0,1]^n$ to $[0,1]$ with the usual product, or the $PMV_f$-algebra of boolean algebras with product defined by the infimum. In this representation, the semi-low $f_u$-ring associated to the boolean algebra $2^n$ is precisely the ring $\Z^n$. 
	
	In  section \ref{mv}, the preliminary concepts about  $MV$-algebras are presented. In section \ref{mvpro}, the $MV$-algebras with product are defined, and in that context, the varieties of $PMV_1$, $PMV_f$, $PMV$-algebras and $MVW$-rigs. Some properties of the  $MVW$-rigs are presented, with examples that illustrate the independence of the axioms chosen. Besides, it is shown that the inclusions between the categories are strict. In section \ref{luanill}  the semi-low $l_u$-rings are presented, (Definition \ref{anillo-especial}) as well as one of the key results of this paper, Theorem \ref{cnumeral}, where the distributive property of the product for $PMV_f$-chains is proven. In section \ref{equivcat} we find the main result of this paper, the construction of the equivalence is extended to the category of  $PMV_f$-algebras with product, and the category of semi-low $f_u$-rings. In section \ref{pmv-anil}, some consequences of the equivalence are drawn, and in particular the construction of the ring associated to the boolean algebras is sketched. Finally, in section \ref{co-extensividad}, we proof that the categories $\mathcal{PMV}_1$ and $\mathcal{PMV}_f$ are coextensive.

	\section{$MV$-algebras}\label{mv}
	Some properties of the theory of  $MV$-algebras are presented, that are relevant to this work. The reader can find more complete information in \cite{Cignoli}.
	
	\begin{Def}[$MV$-algebra]\label{MV}
		An $MV$-algebra is a structure $\left(A,\oplus,\neg,0\right)$ such that $\left(A,\oplus,0\right)$ is a commutative monoid and the operation  $\neg$ satisfies:
		\begin{itemize}
			\item[\textit{i})] $\neg(\neg x)=x,$
			\item[\textit{ii})] $x\oplus \neg 0=\neg 0,$
			\item[\textit{iii})] $\neg(\neg x\oplus y)\oplus y =\neg(\neg y\oplus x)\oplus x.$
		\end{itemize}
	\end{Def}
	Because of properties of $MV$-algebras, $ 0 \leq a \leq u$  for all $a \in A$, with $u = \neg 0$. The operation $\neg$ is called \textbf{negation}, while the operation $\oplus$ is called \textbf{sum}.
	
	\begin{Afir}[Order]\label{orden} Every $MV$-algebra $A$ is ordered by the relation,
		\begin{center}
			$x\leq y$ if and only if $x\ominus y=0$, for all  $x,y\in A.$	
		\end{center}
	\end{Afir}	
	
	\begin{Def}[Homomorphism]
		Given two $MV$-algebras $A$ and $B$, a function  $f: A \to B$ is a homomorphism of $MV$-algebras  if for every $x,y$ en $A:$
		\begin{itemize}
			\item [\textit{i})]  $f(0)=0,$
			\item [\textit{ii})]  $ f(x \oplus y)=f(x) \oplus f(y),$ 
			\item [\textit{iii})] $ f(\neg x) =\neg f(x).$
		\end{itemize}  
	\end{Def}
	
	\begin{Def}[Ideal of an $MV$-algebra]\label{idealmv}
		A non-empty subset $I$ of an $MV$-algebra $A,$ is an ideal if and only if:
		\begin{itemize}
			\item [\textit{i})]  If $a \leq b$ and $b \in I$, then $a \in I.$
			\item [\textit{ii})]  If $a,b \in I$, then $a \oplus b \in I.$
		\end{itemize}
		
		The set of all ideals of the $MV$-algebra $A$ will be denoted by  $Id(A).$ \\
	\end{Def}
	
	\begin{Def}[Prime ideal of an $MV$-algebra]\label{primomv}
		An ideal $P$ of an $MV$-algebra $A$, is prime if for all $a,b \in A$, $a \wedge b \in P$ implies $a \in P$ or $b \in P.$ 
		
		The set of all prime ideals of the $MV$-algebra $A$ will be called ${Spec(A)},$ the spectrum of $A$. 
	\end{Def}
	\begin{Teo}[Chang representation theorem \cite{Chang2}]\label{repchang}
		Every non trivial $MV$-algebra is isomorphic to a subdirect product of $MV$-chains. 
	\end{Teo}
	
	\section{$MV$-algebras with product}\label{mvpro}
	\begin{Def}
		An $MV$-algebra with product is a structure  $(A,\oplus,\cdot, \neg,0)$  such that $(A,\oplus,\neg,0)$ is an $MV$-algebra, and $(A,\cdot)$ is a  semigroup. 
	\end{Def}
	The operation $\cdot $ is called \textbf{product}, and the notation used is : $\underbrace{a\cdot a\cdot \ldots \cdot a}_{n-times}=a^n.$
	
	Next, four varieties of $MV$-algebras with product are defined, namely the $MVW$-rigs, the $PMV$-algebras, the $PMV_f$-algebras and the unitary $PMV_1$-algebras. Some of their properties are proved and in particular, we show that each one is contained in the other.  
	
	From here on, all products are supposed to be commutative. 
	
	\begin{Def}(MVW-rig [\cite{Estrada},2.4])\label{DefMVWrig}
		An \textbf{MVW-rig}   $( A, \oplus,  \cdot, \neg ,0 )$ is an $MV$-algebra with product such that 
		\begin{itemize}
			\item[$i)$] $ a  0 = 0 a = 0,$\label{MVW-i}
			\item[$ii)$]  $\big(a(b\oplus c)\big)\ominus (ab\oplus ac)=0,$\label{MVW-ii}
			\item[$iii)$] $(ab\ominus ac)\ominus\big(a(b\ominus c)\big) =0.$\label{MVW-iii}
		\end{itemize}
	\end{Def}
	
	\begin{Obs} For every  $a,b,c\in A,$  axiom $ii)$ is equivalent to 	
		\begin{equation}\label{reescritura1}
		a(b\oplus c)\leq ab\oplus ac
		\end{equation}
		
		and axiom  $iii)$ is equivalent to 
		\begin{equation}\label{reescritura2}
		ab\ominus ac\leq a(b\ominus c).
		\end{equation}
	\end{Obs}
	\begin{Def}
		An  $MVW$-rig $A$ is called unitary if there exists an element  $s$ with the property that for every  $x$ in $A$ $sx = xs = x$. It is follows that $s$ is unique.
	\end{Def}

	\begin{Def}[$PMV$ \cite{DiNola2}]\label{PMV}
		A $PMV$-algebra $A$, is an $MV$-algebra with product such that for every $a,b,c \in A$ $i)$ $a\odot b=0$ implies $ac \odot bc = 0$; $ii)$ $a \odot b=0$ implies $c(a\oplus b)= ca \oplus cb.$  
	\end{Def}
	In Theorem 3.1 of \cite{DiNola2}, it is shown that the class $PMV$ is equationally definible.
	\begin{Def}[$PMV_f$]\label{PMVE}
		A $PMV_f$-algebra is an $MVW$-rig such that for every $a,b,c \in A$, $ab \leq a \wedge b$, and $a(b \ominus c)=ab \ominus ac$.
	\end{Def}
	\begin{Def}[$PMV$-Unitary algebra \cite{Montagna2}]\label{Unitaria}
		A $PMV$- unitary algebra is an $MV$-algebra $A$ with product such that for every $a,b,c \in A$, $au = a$, y $a(b \ominus c)=ab \ominus ac.$
	\end{Def}
	\begin{Teo}\label{contenencias}
		The following inclusions hold:  $$PMV_1 \subset PMV_f \subseteq PMV \subset MVW\text{-rig}$$
	\end{Teo}
	\begin{proof}
		The first inclusion, $PMV_1 \subset PMV_f,$ follows from Lemma 2.9-iii on \cite{Montagna2} and example \ref{non-unitary-example}.
		
		For the second inclusion, given $a,b,c\in PMV_f,$ if $a \odot b =0$, since  $ac\leq a$ and $bc\leq b$ then  $ac \odot bc \leq a \odot b =0$. On the other hand, $a\odot b=0$ implies $a \leq u \ominus b$ and therefore $ca \leq c(u \ominus b)\leq cu,$  and this last inequality implies (see proposition 2.7-vii, \cite{Montagna2}) $c(b\oplus a)= c(u\ominus ((u \ominus b) \ominus a))=cu \ominus (c(u\ominus b)\ominus ca)= (cu \ominus c(u\ominus b)) \oplus ca = cb \oplus ca$. 
		
		The inclusion $PMV \subset MVW$-rig  is proven in proposition \ref{dinoale}. To see that it is a strict inclusion, see example  \ref{Z-n}.
	\end{proof}
	
	\subsection{Examples and properties of the $MVW$-rigs}
	\begin{Exp}
		Every $MV$-algebra with the  product defined by $ab=0$, for all  $a,b \in A,$ is an $MVW$-rig.
	\end{Exp}
	\begin{Exp}\label{Ejemplousual}
		The $MV$-algebra $[0,1]$ with the usual multiplication inherited from $\mathbb{R}$ is a commutative $MVW$-rig with unitary element $u=1$.
	\end{Exp}
	\begin{Exp}\label{Ejemplosemiusual}
		The $MV$-algebra $[0,u]$ of real numbers with $0 \leq u < 1$ is a commutative $MVW$-rig, but it is not unitary. 
	\end{Exp}
	\begin{Exp}\label{EjemploFree}
		The $MV$-algebra of the continuous functions from $[0,1]^n$ to $[0,1]$ with the usual product for functions is an  $MVW$-rig with the property:  $xy\leq x \wedge y$.
	\end{Exp}
	\begin{Exp}\label{EjemploLukasiewicz}[\cite{Estrada},2.10].
		Consider the algebra $\widetilde{{\L_n}} =  \{ \frac{m}{n^k} \in \mathbb{Q}\cap [0,1]| $  $k, m \in \mathbb{N}\} $ obtained by closing the \L ukasie\-wicz algebra $\L_n$ under products, where   $\L_n = \left\langle\{0,\frac{1}{n-1},\frac{2}{n-1},\cdots,\frac{n-2}{n-1}, 1\}, \oplus, \neg \right\rangle $ with the usual product, is an  $MVW$-rig .
	\end{Exp}
	\begin{Exp}\label{Z-n}[\cite{Estrada},2.11]. $\mathbf{Z}_{n} = \{ 0, 1, \dots , n \}$ with $n \in \mathbb{N}$, $u = n$ as strong unit, $x \oplus y = $ min$\{n, x + y \} $, $\neg x = n-x$ y $x y = $ min$\{n, x \cdot y \} $, is an $MVW$-rig  where sum and product are the usual operations on the natural numbers.  
		
		$\mathbf{Z}_{n}$ is unitary, and $u \neq 1$. The cancellation law does not hold, because the product of two elements can be larger than the supremum. In some cases, the strict inequality in (\ref{reescritura2}) holds, even though the equality (\ref{reescritura1}) is always true. 
		For example, in $\mathbf{Z}_{10}$, $2(7\ominus 6)=2[\neg(\neg 7\oplus 6)]=2[\neg(3\oplus 6)]=2[\neg 9]=2(1)=2 > (2)(7)\ominus(2)(6)=10\ominus 10=0.$
		
		$\mathbf{Z}_{n}$ is not always a $PMV$ algebra either, because for example in $\mathbf{Z}_{10},$ $(3)(2)\odot(3) (2)=6\odot 6=2,$ even though $2\odot 2=0.$
	\end{Exp}
	\begin{Exp}
		$\widehat{\L}_{n+1}=  \left\langle \{0,\frac{1}{n},\cdots \frac{n-1}{n},1\},\oplus, \cdot, \neg, 0,1\right\rangle,$ with product defined by  $\dfrac{x}{n}\cdot \dfrac{y}{n}=\dfrac{min\{n,x\cdot y\}}{n}$, for  $x,y\in \{0, 1, \cdots, n\},$ is an MVW-rig isomorphic to  $\mathbf{Z}_{n}$, and the isomorphism is defined by $\varphi: \L_{n+1} \longrightarrow \mathbf{Z}_{n}, \,\frac{x}{n} \longmapsto x$.
	\end{Exp}
	\begin{Pro}\label{proinfimo}
		Every  $MV$-algebra $A$ with product defined by the infimum  $x\cdot y=x\wedge y$ is an $MVW$-rig.
	\end{Pro}
	\begin{proof}
		Because the product is defined in terms of the order, by the Chang representation theorem, it is enough to show that the result holds for every totally ordered  $PMV$ algebra. Since the product defined by the infimum is associative and commutative, it is enough to prove the inequalities  (\ref{reescritura1}) and (\ref{reescritura2}). Consider $a,b,c \in A$ an $MV$-chain. If $b\oplus c \leq a$ then $b\oplus c=a\wedge (b\oplus c) \leq b\oplus c = (a \wedge b)\oplus (a\wedge c)$. If on the contrary $a\leq b\oplus c$, and $a\le b,c$, then $a=a\wedge (b\oplus c) \leq a \oplus a = (a\wedge b)\oplus (a\wedge c)$. If $a\leq b\oplus c$ and $b\leq a\leq c$, then $a=a\wedge (b\oplus c)\leq a\oplus a = (a\wedge b)\oplus (a\wedge c)$. Similarly it can be proved that $a\wedge (b\ominus c)\geq a\wedge b \ominus a\wedge c.$
	\end{proof}
	\begin{Obs}\label{nopmv}
		Even though every $MV$-algebra is an $MVW$-rig with the product given by the infimum, in general it is not a  $PMV$-algebra, as is shown in the next example. 
	\end{Obs}
	The \L ukasiewics $MV$-algebra  $\L_4$ with the product defined by the infimum is not a $PMV$ algebra because $\dfrac{1}{3}\odot\dfrac{1}{3}=0$  and  $\dfrac{1}{3}=\dfrac{1}{3}\wedge \left(\dfrac{1}{3}\oplus \dfrac{1}{3}\right)<\dfrac{1}{3}\oplus\dfrac{1}{3}=\dfrac{2}{3}.$ 
	
	\begin{Pro}\label{prosupremo}
		Every $MV$-algebra $A$ with product defined by the supremum for non-zero elements, namely,  $ab=a\vee b$, si $a\neq 0$ y $b\neq 0$ and zero otherwise, is an $MVW$-rig.
	\end{Pro}
	\begin{proof}
		It is enough to show (\ref{reescritura1}) and (\ref{reescritura2}) for totally ordered  $MVW$-rigs.
	\end{proof}
	
	\begin{Exp}
		An interesting and relevant particular case of the proposition is when $A$ is a boolean algebra.  A boolean algebra $A$ can be considered as an  $MV$-algebra, where the sum is given by the supremum and negation is the complement. If the product is defined as in the propositions \ref{proinfimo} or \ref{prosupremo}, every boolean algebra is naturally an $MVW$-rig.
	\end{Exp}

	\begin{Pro} Axiom $iii)$  in the definition \ref{DefMVWrig}, is independent of the other axioms for $MVW$-rig. Similarly, axiom  $i)$ is independent of the others. 
	\end{Pro}	
	\begin{proof} 	
		Consider the \L ukasiewicz $MV$-algebra $\L_4$ with product defined by : 
		$$
		a\cdot b=\left\{\begin{array}{ccl}
		0,&si&a=0 \,\, o\,\, b=0,\\
		a\oplus b,&if &a\odot b=0,\\
		a\odot b, &if & a\odot b\neq 0.
		\end{array}\right.
		$$ 
		In this structure axiom $iii)$ does not hold, but the others do. The product is equivalent to the sum on the integers mod  $3$, $\mathbb{Z}_3$ for the elements of $\L _4 - \{0\}$. Therefore, this product is associative and commutative.  
		
		On the other hand, every $MV$-algebra with the supremum as product is a model for all the axioms of $MVW$-rigs, except for  axiom $i)$. The proof is similar to the one given in proposition \ref{prosupremo}. 
	\end{proof}
	
	\begin{Pro}\label{propiedades}[\cite{Estrada},2.5].
		For every $a,b,c \in A$ a commutative  $MVW$-rig the following properties hold:
		\begin{itemize}
			\item[$i)$] If $a \leq b$ then $ac \leq bc,$
			\item[$ii)$] $u^2 \leq u,
			$
			\item[$iii)$] $a\leq b$ and $c\leq d$ then $ac\leq bd.$
		\end{itemize}

	\end{Pro}
	\begin{proof}
		Property $iii)$ follows from $i)$; in fact, $a\leq b$ and $c\leq d$ imply that  $ac\leq bc$  and $cb\leq db$.  
	\end{proof}
	
	\subsubsection{Ideals and Homomorphisms of $MVW$-rigs}
	\begin{Def} Given  $A,B$ $MVW$-rigs, a function $f: A\rightarrow B$ is a homomorphism of $MVW$-rigs in and only if 
		\begin{itemize}
			\item[$i)$] $f$ is a homomorphism of $MV$-algebras and
			\item[$ii)$] $f(ab)=f(a)f(b).$\\
		\end{itemize}
	\end{Def}
	
	\begin{Def} The kernel of a homomorphism $\varphi:A\rightarrow B$ of $MVW$-rigs is $$ker(\varphi):=\varphi^{-1}(0)=\{x\in A|\varphi(x)=0\}.$$
	\end{Def}
	
	\begin{Def}\label{ideal} An ideal of an $MVW$-rig $A$ is a subset $I$ of $A$ that has the following properties: 
		\begin{itemize}
			\item[$i)$] $I$ is an ideal of the subjacent $MV$-algebra $A$.
			\item[$ii)$]Given $a\in I$, and  $b\in A,$ $ab\in I$ (Absorbent Property).
		\end{itemize}
		
		$Id_{_W}(A)$ denotes the set of all ideals of the $MVW$-rig $A$.\\
	\end{Def}
	\begin{Exp}[Boolena Algebras]\label{Boole} Every boolean algebra is an $MVW$-rig, taking the supremum as the sum and the infimum as the product.  The ideals of this  $MVW$-rig are the ideals of the $MV$-algebra, that are at the same time ideals for the lattice. 
	\end{Exp}
	\begin{Obs}
		Note that in proposition \ref{prosupremo}, the $MVW$-rig has no proper non trivial ideals. Its only ideals are zero and the $MVW$-rig.
	\end{Obs}

	
	\begin{Def}[Prime ideal of an $MVW$-rig]\label{primo} An ideal $P$ of an $MVW$-rig, is called prime if for every $a,b\in A,$ $ab\in P$ implies $a\in P$ o $b\in P.$
		
		The set of all prime ideals of the $MVW$-rig $A$ is denoted  $Spec_{_W}(A).$ 
	\end{Def}
	\begin{Pro}\cite{Cignoli, Estrada}. \label{congruencias}  There is a bijective correspondence between the set of all ideals of an  $MVW$-rig  $A$ and the set of its congruences. Namely, given  $I$ an ideal of the $MVW$-rig $A$,  the binary relation defined by $x\equiv_I y$ if and only if $(x\ominus y)\oplus (y\ominus x)\in I$ is a congruence relation, and given $\equiv$ any congruence relation in $A$ the set  $\{x\in A|x\equiv 0\}$ is an ideal of $A$.
	\end{Pro}
	Because it is relevant, the proof of the compatibility of the product is reproduced. The full proof can be found on [\cite{Estrada},2.29].
	\begin{proof}
		Since  $A$ is an $MV$-algebra and $I$ is an $MV$-ideal, $a\equiv_I b$ and $c\equiv_I d$ imply $a\oplus c\equiv_I b\oplus d$ and $\neg a\equiv_I \neg b,$ (see [\cite{Cignoli},1.2.6]). It is left then to prove that $a\equiv_I b$ and $c\equiv_I d$ imply $a c\equiv_I b d.$
		
		Given $a\equiv_I b$ and $c\equiv_I d$ then $a\ominus b\in I$ and $c\ominus d\in I$ respectively, so
		$ac\leq (a\vee b)(c\vee d)=\big((a\ominus b)\oplus b\big)\big((c\ominus d)\oplus d\big)
		\leq(a\ominus b)\big((c\ominus d)\oplus d\big)\oplus b\big((c\ominus d)\oplus d\big)
		\leq (a\ominus b)(c\ominus d)\oplus (a\ominus b)d \oplus b(c\ominus d)\oplus bd.$
		Equivalently
		$$
		ac\ominus bd \leq(a\ominus b)(c\ominus d)\oplus (a\ominus b)d \oplus b(c\ominus d)\in I,
		$$
		because $(a\ominus b)$ and $(c\ominus d)\in I$ and $I$ is absorbent,  $ac\ominus bd\in I.$ Similarly $bd\ominus ac\in I,$ then $(ac\ominus bd)\oplus (bd\ominus ac)\in I,$ and therefore $ac\equiv_I bd.$ 
	\end{proof}
	
	\begin{Obs}\label{clases} For  $a\in A,$ the equivalent class of $a$ respect to $\equiv_I$ will be denoted by $[a]_I$ and the quotient set  $A/\equiv_I$ by $A/I.$   
	\end{Obs}	
	
	Since $\equiv_I$ is a congruence, the operations $\neg[a]_I=[\neg a]_I,$ $[a]_I\oplus[b]_I=[a\oplus b]_I$ y $[a]_I[b]_I=[ab]_I,$ are well defined over $A/I.$
	\begin{Pro}\label{cociente}[\cite{Estrada},2.31]. If $I \in Id_{_W}(A)$, then  $A/I$ is an $MVW$-rig.
	\end{Pro}
	\begin{Cor} Consider $I \in Spec(A)$ and $A$ an $MVW$-rig. If $I$ is absorbent, then $A/I$ is a totally ordered $MVW$-rig. 
	\end{Cor}
	
	\subsection{Examples and properties of the $PMV_f$-algebras}
	
	\begin{Exp}
		The $MV$-algebra $[0,1]$ with the usual multiplication inherited from $\R$ is a $PMV_f.$
	\end{Exp}
	\begin{Exp}\label{non-unitary-example}
		The $MV$-algebra $[0,u]$ with the usual multiplication , and  $0<u<1$,  is a non-unitary $PMV_f.$ 
	\end{Exp}
	\begin{Exp}
		The set of all continuous functions from $[0,1]^n$ to $[0,1]$ with truncated sum and the usual multiplication is a  $PMV_f.$
	\end{Exp}
	\begin{Exp}
		$F[x_1, \cdots, x_n]$ the set of all continuous functions from $[0,1]^n$ to $[0,1]$, that are constituted by  finite polynomials in $\Z[x_1, \cdots, x_n]$, namely, $f(x_1, \cdots, x_n) \in F[x_1,\cdots, x_n] \Leftrightarrow  \exists p_1,\cdots, p_k \in \Z[x_1,\cdots, x_n]$ such that for all $z\in [0,1]^n, f(z)=p_i(z)$, for some $i\in \{1,\cdots,n\}$, is a $PMV_f$. 
	\end{Exp}
	\begin{Exp}
		Every boolean algebra, as in the example \ref{Boole} is a $PMV_f.$
	\end{Exp}
	\begin{Exp}
		Every $MV$-algebra with multiplication defined by the infimum is an $MVW$-rig, not necessarily $PMV_f$-algebra, as was established on the example \ref{proinfimo}. 
	\end{Exp}
	
	\begin{Pro}\label{espectros}
		Given a $PMV_f$-algebra $A$, $Id_W(A)=Id(A)$ and furthermore $Spec_W(A)\subseteq Spec(A)$.
	\end{Pro}
	\begin{proof}
		By definition,  $Id_W(A) \subseteq Id(A)$. Additionally, given  $I \in Id(A)$ and $a\in I$, for all $c\in A$, $ac \leq a \wedge c \in I$, so $I \in Id_W(A)$. On the other hand, given $I \in Spec_{_W}(A)$ and $a,b \in A$ such that $a\wedge b \in I$, then $ab \in I$ because $ab \leq a\wedge b$. Consequently $a\in I$ or $b\in I$, therefore $I \in Spec(A)$.
	\end{proof}
	\begin{Pro}\label{A/I} If $I \in  Id_{_W}(A)$ and $A$  is a $PMV_f$-algebra, then  $A/I$ is a $PMV_f$-algebra.
	\end{Pro}
	\begin{proof}
		It follows directly from proposition \ref{congruencias} and proposition \ref{cociente}.
	\end{proof}
	
	\section{$l_u$-rings}\label{luanill}
	
	\begin{Def}[$l$-group \cite{Birkhoff,Cignoli}] A $l$-group $G$ is a lattice abelian group $\left(G,+,-,0\right)$, such that, the order $<$ is compatible with the sum. 
	\end{Def}
	\begin{Def}For each $x$ in an $l$-group $G,$  its absolute value is defined by $|x|=x^{+}+x^{-},$ where $x^+=x\vee 0$ is the positive part of $x$ and $x^-=-x\vee 0$ is the negative part. 
		
	\end{Def}
	\begin{Def} A strong unit $u$ of an $l$-group $G$ is an element $u$ such that $0\leq u\in G$	and for all $x\in G$ there exists an integer $n\geq 0$ with $|x|\leq nu$.
	\end{Def}
	An $l$-group with strong unit $u$ will be called an $l_u$-group.  
	\begin{Def}[$l$-ideal] An $l$-ideal of an $l$-group $G$ is a subgroup $J$ of $G$ that satisfies:  if $x\in J$ and $|y|\leq |x|$ then $y\in J.$
	\end{Def} 
	\begin{Def}[$l$-prime ideal]
		An $l$-ideal $P$ of an $l_u$-group $G$, is prime if and only if $G/P$ is a chain. 
		
		The set of all $l$-prime ideals of $G$ is called the spectrum of $G$ and denoted by $Spec_g(G).$ 
	\end{Def}
	\begin{Def}\label{lu}[\cite{Birkhoff},XVII.1]. An $l_u$-ring is a ring $R=(|R|, +,\cdot,\leq, u)$ such that $\left\langle R,+,\leq, u \right\rangle$ is an $l_u$-group and, $0\leq x,$ $0\leq y$ implies $0\leq xy,$ where $|R|$ denotes the subjacent set. 
	\end{Def}
	
	From this point on, all rings will be assumed to be commutative. 
	\begin{Def}\label{l-ideal}($L$-ideal [\cite{Birkhoff},XVII.3]). An $L$-ideal $I$ from an  $l$-ring $R$ is an $l$-ideal such that for every $y \in I$ and $x\in R$, $xy\in I.$   $I$ is called irreducible if and only if  $R/I$ is totally ordered. 
		
		The set of all $L$-ideals of $R$ is called $Id(R)$ and the set of all $l$-ideals of the subjacent group is called $Id_{g}(R).$  
	\end{Def}
	\begin{Def}[Low $l$-ring \cite{Montagna}] An $l$-ring is called low if and only if, for all $x,y\geq 0\in R$ we have that $xy\leq x\wedge y.$
	\end{Def}	
	\begin{Def}[Semi-low $l_u$-ring] \label{anillo-especial}An $l_u$-ring $R$ is semi-low if and only if,  for all $a,b\in [0,u]$, $ab\leq a\wedge b.$ 
	\end{Def}
	\begin{Teo}\label{generado} Given $R$ an $l$-ring and $u\in R,$ $u>0,$ and the segment  $\left[0,u\right]=\left\{a\in R\mid 0\leq a\leq u\right\}$ with $\left[0,u\right]^\sharp\subset R,$ the subring generated by $\left[0,u\right],$ then:
		\item [a)] For every  $A\subset[0,u],$ a  $PMV_f,$   $A^\sharp\subset [0,u]^\sharp,$ the subring generated by $A$,   $A^\sharp$ is a semi-low $l_u$-ring with strong unit $u$ and $A=\Gamma(A^\sharp, u).$  
		\item [b)] Every semi-low $l_u$-ring is generated by its segments,  $$[0,u]^\sharp=\left\{x\in R\mid \exists n\geq 0, |x|\leq nu\right\}.$$
	\end{Teo}
	
	\begin{proof}
		\begin{itemize}
			\item [a)] Given $A=\left\langle|A|, \oplus, \cdot, \neg, 0 \right\rangle$ a $PMV_f$, then $A=\left\langle|A|, \oplus, \neg, 0 \right\rangle$ is an $MV$-algebra; call $A^*$ the associated $l_u$-group. Then the subjacent sets are equal $|A^\sharp|=|A^*|$, because for every $a \in A^\sharp$, $a= \sum \epsilon_i b_ic_i +\sum \delta_j d_j$, with $b_i, c_i, b_ic_i, d_j \in A$ and $\epsilon_i, \delta_j \in \{1,-1\} $,  is a sum of elements of $A$. Therefore $A=\Gamma(A^\sharp,u)$ because of  Theorem 1.2-a) of \cite{Yuri3}. On the other hand, $x,y \in A^\sharp \cap [0,u]$ implies $x,y \in A$ y $xy \leq x \wedge y$.
			\item [b)] From theorem 1.2-b) of \cite{Yuri3}, it follows that  $$[0,u]^*=\left\{x\in R\mid \exists n\geq 0, |x|\leq nu\right\}$$ is an $l_u$-group with strong unit  $u,$ and for the reasons exposed above, the subjacent sets are  $|R|=\left|[0,u]^*\right|=|[0,u]^\sharp|$, so  $R= [0,u]^\sharp$.
		\end{itemize} 
	\end{proof}
	
	\section{Equivalence between the categories $\mathcal{PMV}_f$ and $\mathcal{LR}_u.$}\label{equivcat}
	
	\begin{Def} We called  $\mathcal{PMV}_f$  and  $\mathcal{CPMV}_f$ to the categories whose objects are $PMV_f$-algebras and $PMV_f$-chains and homomorphisms between them respectively.
	\end{Def}
	
	\begin{Def} We called $\mathcal{LR}_u,$ and $\mathcal{C}\mathcal{LR}_u$ to the categories whose objects are semi-low $l_u$-rings and chain semi-low $l_u$-rings, and homomorphisms between them respectively. 
	\end{Def}
	
	\subsection{Categorical Equivalence between $\mathcal{CPMV}_f$ and  $\mathcal{C}\mathcal{LR}_u$}
	\begin{Teo}[Chang's construction of the $l_u$-group $A^*$ \cite{Chang2}, Lemma 5]\label{ecuchang}
		Given $A$ an $MV$-chain, $A^*=\left\langle \mathbb{Z}\times A,+,\leq \right\rangle$ together with the operations $(m+1, 0)=(m,u),$ \,\, $(-m-1,\neg a)=-(m,a),$ and  $(m,a)+(n,b)=(m+n,a\oplus b)$ if $a\oplus b\neq u$, or $(m,a)+(n,b)=(m+n+1,a\odot b)$ if $a \oplus b = u,$ is a chain $l_u$-group with strong unit $\mathbf{u}=(1,0)=(0,u),$ where the order $\leq$ is given by $(m,a)\leq(n,b)$ if and only if $m<n$ or $m=n$ and $a\leq b$.
	\end{Teo}
	
	\begin{Obs}\label{universo} Denote $|A^*|=\mathbb{Z}\times A.$
	\end{Obs}
	\subsubsection{The functor $(-)^\sharp \colon \mathcal{CPMV}_f \to \mathcal{C}\mathcal{LR}_u$} 
	\begin{Def}\label{defanillo}
		Given $A$ a $PMV_f$-chain, the structure $A^\sharp$ is define following the Chang's construction \cite{Chang2} as follows. $A^\sharp=\left\langle |A^*|,+,\cdot,\leq\right\rangle,$ with $ \left\langle |A^*|,+, \mathbf{u}, \leq\right\rangle,$ its associated $l_u$-group where the product is defined by $$ (m,a)\cdot (n,b):=mn(0,u^2)+m(0,bu)+n(0,au)+(0,ab).$$\\ with  $n(0,x)=\underbrace{(0,x)+\ldots +(0,x)}_{n-times}$ for $n\geq 0$ and $n(0,x)=\underbrace{-(0,x)-\ldots -(0,x)}_{n-times}$ for $n<0.$
	\end{Def}
	\begin{Pro} 
		The product defined above is well defined. 
	\end{Pro}
	\begin{proof}
		Note that in $A^*,$  $(m+1,0)=(m,u);$ it is enough then to observe directly from the definition of the product that  $(m,u)\cdot (n,b)= (m+1, 0)\cdot (n,b)$.
	\end{proof}
	\begin{Afir}\label{disodot}
		For every $x,y, z \in A$, $$(0,x)(0,y\odot z)=(0,xy)+(0,xz)-(0,xu).$$
	\end{Afir}
	\begin{proof}
		The equality follows directly from theorem \ref{contenencias} and definition \ref{defanillo}, if $y \odot z =0$. On the other hand,   $y\odot z\neq 0 \Longleftrightarrow \neg y\oplus \neg z\neq u,\Longleftrightarrow \neg y\odot \neg z=0,$ implies 	
		$$\begin{array}{rcl}
		(0,x)(0,y\odot z)	&=&(0,x)(0,\neg(\neg y\oplus \neg z))\\
		&=& (0,x)\left[-(-1,\neg y\oplus \neg z)\right]\\
		&=&-(0,x)(-1,\neg y\oplus \neg z)\\
		&=& -(0,x)\left[(0,\neg y)+(-1,\neg z)\right]\\
		&=& -(0,x)(0,\neg y)-(0,x)(-1,\neg z)\\
		&=& (0,x)\left[-(0,\neg y)\right]+(0,x)\left[-(-1,\neg z)\right]\\
		&=&(0,x)(-1,y)+(0,x)(0,z)\\
		&=&(0,xy)-(0,xu)+(0,xz).
		\end{array}$$
		
	\end{proof}
	\begin{Teo}\label{cnumeral}
		Given  $A$ a $CPMV_f$, $A^\sharp$ is a chain semi-low $l_u$-ring. 
	\end{Teo}
	\begin{proof}
		It is clear that  $A^\sharp$ with the sum operation and the associated order is a chain $l_u$-group. It is enough to show that with the product given by definition \ref{defanillo}, it is a semi-low ring. 
		
		For every $(m,x),(n,y),(s,z)\in A^\sharp,$ the following properties hold:
		
		\textbf{ Distributivity}
		$$(m,x)[(n,y)+(s,z)]=(m,x)(n,y)+(m,x)(s,z).$$
		Because of the theorems \ref{contenencias} and \ref{ecuchang}, $z\odot y = 0$ implies $xz \odot xy=0$, and  $x(z \oplus y)= xz \oplus xy$, so, $(m,z)+(n,y)=(m+n, z\oplus y)$ and $(m,xz)+(n,xy)=(m+n, xz\oplus xy)$.
		
		This affirmation will be proved dividing the proof in two cases.
		
		\textbf{Case 1.}  $y\odot z=0.$\\[7pt]
		$\begin{array}{rcl}
		(m,x)[(n,y)+(s,z)]&=&(m,x)[(n+s,y\oplus z)]\\[5pt]
		&=&m(n+s)(0,u^2)+m(0,(y\oplus z)u)+(n+s)(0,xu)\\[5pt]
		&&+(0,x(y\oplus z))\\[5pt]
		&=&mn(0,u^2)+ms(0,u^2)+m(0,yu\oplus zu)+n(0,xu)\\[5pt]
		&&+s(0,xu)+(0,xy\oplus xz)\\[5pt]
		&=&mn(0,u^2)+ms(0,u^2)+m(0,yu)+m(0,zu)\\[5pt]
		&&+n(0,xu)+s(0,xu)+(0,xy)+(0,xz)\\[5pt]
		&=&\big[mn(0,u^2)+m(0,yu)+n(0,xu)+(0,xy)\big]\\[5pt]
		&&+\big[ms(0,u^2)+m(0,zu)+s(0,xu)+(0,xz)\big]\\
		&=&(m,x)(n,y)+(m,x)(s,z).\\[10pt]
		\end{array}$
		
		From theorem \ref{ecuchang},  $y\odot z\neq 0$  implies, $(n,y)+(s,z)=(n+s+1, y\odot z)$, and form affirmation \ref{disodot}, $(0,x)(0,y\odot z)=(0,x(y\odot z))=(0,xy)+(0,xz)-(0,xu).$ 
		
		\textbf{Case 2.} $y\odot z\neq 0$\\[7pt]	
		$\begin{array}{lcl}
		(m,x)[(n,y)+(s,z)]&=&(m,x)[(n+s+1,y\odot z)]\\
		&=&m(n+s+1)(0,u^2)+m(0,(y\odot z)u)\\
		&& +(n+s+1)(0,xu)+(0,x(y\odot z))\\
		&=&m(n+s+1)(0,u^2)+m\big[(0,yu)+(0,zu)-(0,u^2)\big]\\
		&&+(n+s+1)(0,xu)+\big[(0,xy)+(0,xz)-(0,xu)\big]\\
		&=&mn(0,u^2)+ms(0,u^2)+m(0,u^2)+m(0,yu)\\
		&&+m(0,zu)-m(0,u^2) +n(0,xu)+s(0,xu)\\
		&&+(0,xu)+(0,xy)+(0,xz)-(0,xu)\\
		&=&\big[mn(0,u^2)+m(0,yu)+n(0,xu)+(0,xy)\big]\\
		&&+\big[ms(0,u^2)+m(0,zu)+s(0,xu)+(0,xz)\big]\\
		&=&(m,x)(n,y)+(m,x)(s,z).\\
		\end{array}$
		
		\textbf{ Associativity}
		$$\begin{array}{rcl} (m,x)\big[(n,y)(s,z)\big]&=&(m,x)\big[ns(0,u^2)+n(0,zu)+s(0,yu)+(0,yz)\big]\\
		&=&mns(0,u^3)+ns(0,xu^2)+mn(0,zu^2)+n(0,xzu)\\
		&&+ms(0,yu^2)+s(0,xyu)+ m(0,yzu)+(0,xyz) \\
		&=&mns(0,u^3)+mn(0,zu^2)+ms(0,yu^2)+m(0,zyu)\\
		&&+ ns(0,xu^2)+n(0,xuz) + s(0,xyu)+(0,xyz)\\ &=&\big[mn(0,u^2)+m(0,yu)+n(0,xu)+(0,xy)\big](s,z)\\
		&=&\big[(m,x)(n,y)\big](s,z).
		\end{array}$$
		
		Given $(0,0)\leq (m,x)$ and $(0,0)\leq (n,y)$ it is clear that $0\leq m$ and $0\leq n,$ so $$(0,0)\leq mn(0,u^2)+m(0,yu)+n(0,xu)+(0,xy)=(m,x)(n,y).$$
		Now it can be proved that $A^\sharp$ is semi-low.\\ Given $(0,0)\leq(m,x),(n,y)\leq(0,u)$ it must be that  $m=n=0,$ so $$(0,x)(0,y)=(0,xy)\leq (0,x\wedge y)=(0,x)\wedge(0,y).$$
	\end{proof}
	
	\begin{Cor}\label{productocadena}	For every $PMV_f$-chain $A$, 
		$$\displaystyle\left(\sum_{i=1}^{n}(0,x_i)\right)\left(\sum_{i=1}^{n}(0,y_i)\right)= \sum_{i=1}^{n}(0,x_iy_i),$$ in $A^\sharp$.
	\end{Cor}
	\begin{Pro}\label{numfun} $(-)^\sharp$ is functorial.
	\end{Pro}
	For $h:A\rightarrow B$ in $\mathcal{CPMV_E},$ define $h^\sharp:A^\sharp \rightarrow B^\sharp$ en $\mathcal{C}\mathcal{LR}_u$ as follows:
	$h^\sharp (m,a):=(m,h(a)).$ By construction (see [\cite{Yuri3}, 2.2]), $h^\sharp$   is a homomorphism of $l_u$-groups, so it is enough to prove that $h^\sharp$ is a homomorphism of $l_u$-rings. This follows directly from definition \ref{defanillo}. Namely, $h^\sharp[(m,a)(n,b)]=h^\sharp(m,a)h^\sharp(n,b).$
	
	Besides, given  $(m,a)\in A^\sharp,$ $ (gh)^\sharp(m,a)=(m,gh(a))=g^\sharp(m,h(a))=g^\sharp\circ h^\sharp(m,a).$
	
	\subsubsection{The Functor $\Gamma \colon \mathcal{C}\mathcal{LR}_u \to \mathcal{CPMV}_f$}
	
	\begin{Def}
		For $(R,u)$ a chain semi-low $l_u$-ring, define $\Gamma (R,u)=\{ x\in R\mid 0\leq x\leq u\}$
		together with the operations $x\oplus y= (x+y)\wedge u,$ $\neg x= u-x$ and, $x\cdot y =xy.$ The multiplication is well defined because  $xy \leq x\wedge y\leq u.$ 
	\end{Def}
	\begin{Pro}[\cite{DiNola2},3.2]\label{Gamma} Given $(R,u)$ an $l_u$-ring that satisfies $u^2 \leq u$, $\Gamma(R,u)$ is a $PMV$-algebra.
	\end{Pro}
	
	\begin{Obs} If  $R$ is a chain semi-low $l_u$-ring, then $x(y\vee z)=xy\vee xz.$ In fact, it can be assume without loss of generality that  $y\leq z.$ Then  $x(y\vee z)=xz=xy\vee xz$. A similar statement for the infimum is true. Consequently, in this case  $x(y \ominus z)=xy \ominus xz$.
	\end{Obs}
	\begin{Cor}\label{gammacadena}
		For every $R$ a chain semi-low  $l_u$-ring, $\Gamma(R,u)$ is an $\mathcal{PMV}_f$.
	\end{Cor}
	
	\begin{Afir}\label{Gamma-funt} $\Gamma$ is functorial.
	\end{Afir}
	Given $\alpha:(R,u)\rightarrow (H,v)$ in $\mathcal{C}\mathcal{LR}_u,$ define $\Gamma(\alpha):\Gamma(R,u) \rightarrow \Gamma(H,v)$ in $\mathcal{CPMV}_f$ as follows:
	$\Gamma(\alpha):=\alpha|_{[0,u]}.$ By construction  $\Gamma(\alpha)$ is a homomorphism of totally ordered $MV$-algebras. Then it is enough to see that it respects products, that is, $$\Gamma(\alpha)(a)\Gamma(\alpha)(b)=\alpha(a)\alpha(b)=\alpha(ab)=\Gamma(\alpha)(ab).$$ Therefore, $\Gamma(\alpha)$ is a  morphism in  $\mathcal{C}\mathcal{LR}_u,$ such that for all $x\in\Gamma(R,u),$  it holds that $\Gamma(\beta)\Gamma(\alpha)(x)=\Gamma(\beta)\left(\Gamma(\alpha)(x)\right) =\Gamma(\beta)(\alpha(x))=\beta(\alpha(x))= (\beta\alpha)(x)= \Gamma(\beta\alpha)(x).$
	
	\begin{Teo}\label{Equivalencia-cadenas} For every  $PMV_f$-chain $A$ and chain semi-low  $l_u$-ring $(R,u),$ the following are isomorphisms:
		\begin{center}
			$A\cong \Gamma(A^\sharp, u)$ and $R\cong (\Gamma(R,u))^\sharp$
		\end{center}
	\end{Teo}
	
	\begin{proof}
		The correspondences  $i$ and $\upsilon:$
		$$
		\begin{array}{rclrcl}
		i:A&\longrightarrow&\Gamma(A^\sharp,u)& \hspace{30pt}\upsilon:(\Gamma(R,u))^\sharp &\longrightarrow& R\\
		a&\longmapsto& (0,a)& (m,x)&\longmapsto& mu+x\\
		\end{array} $$
		
		are isomorphisms of $MV$-algebras and $l_u$-groups respectively   (\cite{Chang2}, Lemma 6). It is then enough to prove that they respect the product. For $a,b\in A,$ $$i(ab)=(0,ab)=(0,a)(0,b)=i(a)i(b).$$
		
		On the other hand, given $(m,a),(n,b)\in A^\sharp,$ it is true that:\\
		$
		\upsilon\left[(m,a)(n,b)\right]=\upsilon\left[mn(0,u^2)+m(0,bu)+n(0,au)+(0,ab)\right]
		=mn\left[\upsilon(0,u^2)\right]+m\left[\upsilon(0,bu)\right]+n\left[\upsilon(0,bu)\right]+\upsilon(0,ab)
		=mnu^2+mbu+nau+ab
		=mu(nu+b)+a(nu+b)
		=(mu+a)(nu+b)
		=\upsilon(m,a)\upsilon(n,b).$
	\end{proof}
	It is now easy to prove that the isomorphisms defined based on Chang's construction given in theorem \ref{ecuchang}, $i$ and $\upsilon$, determine a categorical equivalence.
	\begin{Teo}\label{equivalencia cadenas}
		The  isomorphisms $i$ and $\upsilon$ defined above are natural transformations associated to the functors $\Gamma(-)^\sharp$ and $(-)^\sharp \Gamma$ respectively and establish an equivalence of categories $$\mathcal{CPMV}_f \stackrel{(-)^\sharp}{\longrightarrow} \mathcal{CLR}_u \hspace{1cm} \mathcal{CLR}_u \stackrel{\Gamma}{\longrightarrow} \mathcal{CPMV}_f$$
	\end{Teo}
	\begin{proof}
		The proof is analogous to theorem 2.2 on  \cite{Yuri3}. Given  $A\stackrel{h} {\longrightarrow}B$ in $\mathcal{CPMV}_f,$ and with $(R,u)\stackrel{\varphi}{\longrightarrow}(H,w)$ in $\mathcal{CLR}_u$, the naturality of $i$ and $v$ follows from the commutativity of the following diagrams:
		
		\centerline{
			\xymatrix@R=0.5cm{
				A\ar[r]^-{i}\ar[d]_{h} & \Gamma(A^\sharp,u)\ar @{^{(}->}[r]\ar[d]^{\Gamma(h^\sharp)} & A^\sharp\ar[d]& & \Gamma(R,u)^\sharp \ar[d]_{(\Gamma\varphi)^\sharp} \ar[r]^{v}& (R,u)\ar[d]^{\varphi} \\  
				B\ar[r]^-{i}&\Gamma(B^\sharp,u )\ar  @{^{(}->}[r]& B^\sharp & & \Gamma(H,w)^\sharp \ar[r]^{v}& (H,w) \\
		}}
		Given $a\in A,$ it holds that $\Gamma(h^\sharp)i(a)=\Gamma(h^\sharp)(0,a)=h^\sharp(0,a)=(0,h(a))=ih(a),$ and for $(n,x)\in \Gamma(R,w)^\sharp,$ then $\varphi v(n,x)=\varphi(nu+x)=nw+\varphi(x)=v(n,\varphi(x))=v(n,(\Gamma\varphi)(x))=v(\Gamma\varphi)^\sharp(n,x).$ 
	\end{proof}
	\subsection{Categorical equivalence between the categories $\mathcal{PMV}_f$ and $\mathcal{LR}_u.$}
	\subsubsection{Subdirect representation of $\mathcal{PMV}_f$-algebras by chains }
	Recall the partial order isomorphism between the ideals of an $l_u$-group $G$ and the ideals of its $MV$-algebra $\Gamma(G,u),$ established on theorem 7.2.2 of \cite{Cignoli}. 
	
	\begin{Teo}\label{isomorfismo-de-l-ideales} Given G an $l_u$-group and $A=\Gamma(G,u),$ the correspondence 
		$$\begin{array}{rcl}
		\phi: \mathcal{I}(A)&\longrightarrow&\mathcal{I}(G)\\
		J&\longmapsto & \phi(J)=\{x\in G| \left|x\right|\wedge u\in J\}\\
		\end{array}
		$$
		is a partial order  isomorphism between the ideals of the  $MV$-algebra $\Gamma(G,u)$ and the $l$-ideals of the $l_u$-group, and its inverse is given by  $H\mapsto \psi(H)=H \cap [0,u].$\\
	\end{Teo}

	\begin{Pro}\label{isomorfismo-de-ideales-2} For a semi-low $l_u$-ring $R$, an ideal $J$ of the $PMV_f$-algebra $\Gamma(R,u)$ and $\phi(J)$ the ideal of the $l_u$-group $(R,+,u)$ as in theorem \ref{isomorfismo-de-l-ideales}, it holds that $\phi(J)=J^\sharp$ with $$J^\sharp=\left\{x\in R \mid x=\sum\limits_{i=1}^{m}\epsilon_ic_i, \,\, c_i\in J,\epsilon_i\in \{-1,1\}\right\}.$$
	\end{Pro}
	\begin{proof}
		$J^\sharp$ is an $l$-ideal of the $l_u$-group $(R,+,u).$ In fact, $J^\sharp$ is a subgroup  of  $R$ by construction. Next it must be proven that given $x\in J^\sharp$ and $y\in R$ such that $|y|\leq |x|,$  $y\in J^\sharp.$ Suppose without loss of generality that  $|x|=x^+=x$ and $|y|=y^+=y.$\\

		Since $x=\sum\limits_{i=1}^{m}\epsilon_{i}c_i$ with $c_i\in J,$ 
		\begin{equation}\label{jnum}
		x\wedge u=\left(\sum\limits_{i=1}^{m}\epsilon_{i}c_i\right)\wedge u=\left|\sum\limits_{i=1}^{m}\epsilon_{i}c_i\right|\wedge u\leq \left(\sum\limits_{i=1}^{m}c_i\right)\wedge u=\oplus_{i=1}^{n}c_i\in J,
		\end{equation}
		therefore, $x\wedge u\in J.$
		
		By theorem 1.5-c of \cite{Yuri3}, it is enough to consider $x=\sum\limits_{k=0}^{n}a_k,$ with $a_k=(x-ku)\wedge u\vee 0,$  where $0<x<nu$ for some $n\in \mathbb{N},$ since the elements are in an $l_u$-group.
		
		If $x-ku>0,$ $a_k\in J$ because  
		$$(x-ku)\wedge u\vee 0=(x-ku)\wedge u\leq x\wedge u\in J,$$  
		
		by inequality (\ref{jnum}).
		If $(x-ku)<0$ then $a_k=0\in J.$
		
		Since $0\leq y\leq x\leq nu$ implies  $b_k=(y-ku)\wedge u\vee 0\leq (x-ku)\wedge u\vee 0= a_k \in J,$ then $y=\sum\limits_{k=0}^{n}b_k \in J^\sharp$.

		The same proof can be used if $x=x^-$ and $y=y^-.$  Because $|x|=x^++x^-$ and $|y|=y^++y^-,$ both are sums of positive elements and  $J^\sharp$ is a subgroup of $R,$ $|y|\leq|x|$ and $x\in J^\sharp$ imply  $y\in J^\sharp.$ 
		
		By construction $J \subseteq J^\sharp$,  and for inequality (\ref{jnum}), $J^\sharp \cap[0,u]\subseteq J$ so $$J^\sharp \cap[0,u]=J=\phi(J)\cap[0,u],$$ thus, by the isomorphism given in theorem \ref{isomorfismo-de-l-ideales}, $J^\sharp =\phi(J).$ 
	\end{proof}
	
	\begin{Cor}\label{fanillo1} $\phi(J)=J^\sharp$ is an ideal of the $l_u$-ring.
	\end{Cor}
	\begin{proof}
		It is enough to show that $J^\sharp$ is absorbent.  For any $r\in R,$  $r=\sum\limits_{j=1}^{m}\alpha_jd_j$ with $d_j\in [0,u]$ and $\alpha_j\in \{-1,1\},$ because of theorem \ref{generado}, and given $x\in J^\sharp,$ $x=\sum\limits_{i=1}^{n}\epsilon_ic_i,$ with $c_i\in J,$  then 
		$$rx=\sum\limits_{j=1}^{m}\alpha_jd_j \sum\limits_{i=1}^{n}\epsilon_ic_i=\sum\limits_{i,j=1}^{mn}\alpha_j\epsilon_i d_jc_i,$$ where $d_jc_i\in J,$   since this is absorbent, therefore $rx\in J^\sharp.$ 
	\end{proof}
	\begin{Cor}\label{l-L}
		In a semi-low $l_u$- ring every $l$-ideal is an $L$-ideal. $$Id_g(R)=Id(R).$$
	\end{Cor}
	\begin{proof}
		For any $J \in Id_g(R),$ because of theorem \ref{isomorfismo-de-l-ideales} it holds that $J \cap [0,u] \in Id(\Gamma(R,u))$. Because $\Gamma(R,u) \in \mathcal{PMV}_f$, $J \cap [0,u]$ absorbs, then $J\cap [0,u] \in Id_{_W}(\Gamma(R,u))$ and consequently by proposition \ref{isomorfismo-de-ideales-2}, $J = (J\cap [0,u])^\sharp\in Id(R)$. 
		In particular,	$Spec_g(R) \subset Id(R)$.
	\end{proof}
	
	\begin{Cor}\label{cocienteanillo}
		For any $J \in Id_g(R)$ where  $R$ is  a semi-low $l_u$-ring, $R/J$ is a semi-low $l_u$-ring. 
	\end{Cor}
	\begin{Teo}\label{cociente-ideales} For any  $J \in Id_g(R)$ where $R$ is a semi-low $l_u$-ring,
		$$\Theta \colon \Gamma(R/J,u_J) \to  \Gamma(R,u)/(J\cap[0,u]); \,\, [x]_J \longmapsto [x]_{J\cap [0,u]}$$
		is an isomorphism of $PMV_f$-algebras.
	\end{Teo}
	\begin{proof}
		Because the $MV$-algebras are isomorphic, due to theorem 7.2.4 of \cite{Cignoli}, it is enough to see that the isomorphism respects products. Using corollary \ref{cocienteanillo}, proposition \ref{Gamma} and the definition of $\Theta$ it follows that $$\Theta([a]_J[b]_J)= \Theta([ab]_J)=[ab]_{J\cap [0,u]}=\left([a]_{J\cap [0,u]}\right) \left([b]_{J\cap [0,u]}\right).$$
	\end{proof}
	\begin{Cor}
		If $J \in Spec_g(R)$ then $\Theta$ is an isomorphism of  $PMV_f$-chains.
	\end{Cor}
	\begin{proof}
		It follows from the last theorem and the corollary \ref{gammacadena}.
	\end{proof}
	\begin{Teo}\label{repsubdirecta}
		Every $PMV_f$-algebra is isomorphic to a subdirect product of $PMV_f$-chains.
	\end{Teo}
	\begin{proof}
		For any $PMV_f$-algebra $A$  there is an injective homomorphism of $MV$-algebras, $$\widehat{(\,)}:A\rightarrow\prod\limits_{P \in Spec(A)} A/P,$$ mapping $a \mapsto \widehat{a}$ where $\widehat{a} \colon Spec\,A \to \bigsqcup\limits_{P \in Spec\, A} A/P  $ with $\widehat{a}(P)= [a]_P.$ It is a homomorphism of $PMV_f$-algebras, and $\pi_P \circ \widehat{(\,)} \colon A \to A/P$ is a surjective homomorphism for each prime ideal $P \in Spec(A)$. In fact, every prime ideal $P$ in the $MV$-algebra is an ideal in the $PMV_f$-algebra, as proven in proposition \ref{espectros}, where $\widehat{ab}= \widehat{a} \cdot \widehat{b},$ due to the correspondence between ideals and  congruences in any $PMV_f$-algebra.
	\end{proof}
	\begin{Cor}
		Every $PMV_f$-equation (see \cite{Cignoli}, section 1.4) that holds in any  $PMV_f$-chain holds in every $PMV_f$-algebra.
	\end{Cor}
	\begin{Cor}\label{mwequation}
		In every  $PMV_f$-algebra it holds that $a(b \wedge c)= ab \wedge ac$,\, $a(b \vee c)= ab \vee ac$.
	\end{Cor}
	\begin{Cor}\label{GammaPMV}
		If $(R,u)$ is a semi-low $l_u$-ring, $\Gamma((R,u))$ is a $PMV_f$-algebra.
	\end{Cor}
	\begin{proof}
		From the corollary \ref{gammacadena} it follows that every $\Gamma(R,u)/P $ is a $PMV_f$-chain for every $P,$ and the other hand, $\Gamma((R,u))$ is isomorphic to a subdirect product of  $\prod\limits_{P \in Spec(\Gamma(R,u))}\Gamma(R,u)/P$
		and therefore $\Gamma((R,u))$ is a $PMV_f$-algebra.
	\end{proof}
	\begin{Teo}
		Every semi-low $l_u$-ring  $R$ is isomorphic to a subdirect product of chains. 
	\end{Teo}
	\begin{proof}
		It is enough to show that the injective homomorphism of $l_u$-groups given by $$\widehat{(-)}^g \colon R \to  \prod\limits_{P \in Spec_g(R)} R/P; \hspace{0.5cm} x \longmapsto [x]_P,$$
		is an $l_u$-ring homomorphism. In fact, from theorem 7.2.2 of \cite{Cignoli} and corollary \ref{fanillo1} it follows directly that $R/P$ is a semi-low $l_u$-ring for every $P \in Spec_g(R)$.
	\end{proof}
	\subsubsection{Extension of the functors $(-)^\sharp$ y $\Gamma$}
	The diagram on the left  will be completed to extend the construction of Chang to the functor $\mathcal{PMV}_f \stackrel{(-)^\sharp}{\longrightarrow} \mathcal{LR}_u$, and then it will be proven that this extends the equivalence from the first row to an equivalence in the second. \\ 
	\centerline{
		\xymatrix@R=0.5cm{
			\mathcal{C}\mathcal{PMV}_f\ar[r]^{(-)^\sharp}\ar@{_{(}->}[d]_{i_\mathcal{M}} &\mathcal{C}\mathcal{R}_{u}\mathcal{E}\ar@{_{(}->}[d]_{i_{\mathcal{R}}} \\
			\mathcal{PMV}_f&\mathcal{LR}_u
		} \hspace{1cm} \xymatrix@R=0.5cm{
			\mathcal{C}\mathcal{LR}_u\ar[r]^{\Gamma}\ar@{_{(}->}[d]_{i_\mathcal{R}} &\mathcal{C}\mathcal{PMV}_f\ar@{_{(}->}[d]_{i_{\mathcal{M}}} \\
			\mathcal{LR}_u\ar[r]^{\Gamma}&\mathcal{PMV}_f
	} }
	
	\begin{Def}
		For any $PMV_f$-algebra $A$, we define $$A^\circ=\{(0,\widehat{a}):a\in A\}\subseteq  \prod\limits_{P\in Spec\, A}(A/P)^\sharp.$$
	\end{Def}
	\begin{Def}[Associate $l_u$-ring]
		For any  $PMV_f$-algebra $A$ we define $A^\sharp= gen(A^\circ)$ as the $l$-ring generated in the $l$-ring $\prod\limits_{P\in Spec\, A}(A/P)^\sharp$.
	\end{Def}
	
	\textbf{Notation.} $|A^*|=\left\{x\in \prod\limits_{P\in Spec\,A}(A/P)^*\mid x=\sum\limits_{i=1}^n \epsilon_{i}(0,\widehat{a_i}), a_i\in A,  n \in \mathbb{N} \right\}.$\\

	\begin{Afir}\label{anum} $A^\sharp$ is a semi-low $l_u$-ring and $A^\sharp = \left\langle |A^*|,+,\cdot,u,\leq\right\rangle$ where  $A^*=\left\langle |A^*|,+,u,\leq\right\rangle$ is the $l_u$-group associated to the subjacent $MV$-algebra $A$, and the product is defined as follows:
		$$\begin{array}{rcl}
		\varphi:|A^*|^2&\longrightarrow& |A^*|\\
		\left(x,y\right) &\longmapsto& \varphi\left(x,y\right):= x\cdot y.
		\end{array}$$
		with $x=\sum\limits_{i=1}^n \epsilon_i (0,\widehat{a_i}),$ $y=\sum\limits_{j=1}^m \delta_j (0,\widehat{b_j}),$\,\, $x\cdot y=\sum\limits_{i,j=1}^{nm}\epsilon_i\delta_k (0,\widehat{a_ib_j}),$ $\epsilon_i,\delta_j\in \{-1,1\}$ y, $a_i,b_j\in A,$. 
	\end{Afir}
	
	\begin{proof}
		$\varphi$ is well defined because for each  $P\in Spec(A)$ the product  $(x\cdot y)(P)=x(P)\cdot y(P)$ coincides with the product given in definition \ref{defanillo} and described in corollary \ref{productocadena}. From theorem \ref{cnumeral} it follows that the operation is associative and distributive.\\
		On the other hand, $\left\langle |A^*|,+,\cdot,\leq\right\rangle$ is a semi-low $l_u$-ring because for every $\mathbf{0}\leq x\, ,y \leq \mathbf{u}$,
		\begin{center} $ x=\sum\limits_{i=1}^n \left(0,\widehat{a_i}\right)=\left(0,\oplus_{i=1}^n \widehat{a_i}\right)$ e $y=\sum\limits_{j=1}^m \left(0,\widehat{b_j}\right)=\left(0,\oplus_{j=1}^m \widehat{b_j}\right),$\end{center}
		
		$(x\cdot y)(P)=x(P)\cdot y(P)
		=\left(0,\left[\oplus a_i\right]_P\right)\cdot\left(0,\left[\oplus b_j\right]_P\right)=\left(0,\left[\oplus a_i\right]_P\cdot\left[\oplus b_j\right]_P\right)
		\leq  \left(0,\left[\oplus a_i\right]_P\wedge\left[\oplus b_j\right]_P\right)
		= \left(0,\left[\oplus a_i\right]_P\right)\wedge\left(0,\left[\oplus b_j\right]_P\right)
		= x(P)\wedge y(P).$\\
		
		Since $A^\circ\subseteq |A^*|$, and every $l_u$-ring  $H$ that contains  $A^\circ,$ must contain all finite sums and products of elements of $A^\circ,$ 
		$\left\langle |A^*|,+,\cdot,u,\leq\right\rangle\subseteq H.$ 
	\end{proof}
	
	\begin{Def}\label{hsharp}
		For any $h:A\rightarrow B$ in the category $\mathcal{PMV}_f,$ we define $h^\sharp:A^\sharp \rightarrow B^\sharp$ en $\mathcal{LR}_u$ by $h^\sharp \left(\sum\limits_{i=1}^n\epsilon_i(0,\widehat{a_i})\right):=\sum\limits_{i=1}^n\epsilon_i\left(0,\widehat{h(a_i)}\right).$
	\end{Def}
	
	\begin{Teo}
		The application $(-)^\sharp \colon \mathcal{PMV}_f \rightarrow \mathcal{LR}_u$ that assigns to each $PMV_f$-algebra $A$ the $l_u$-ring $A^\sharp$, is functorial.
	\end{Teo}
	\begin{proof}
		Since for every $h:A\rightarrow B$ in the category $\mathcal{PMV}_f$	$h^\sharp$ is a homomorphism of $l_u$-rings and $\mathbf{h}^\sharp$ is a homomorphism of $l$-rings such that the following diagram commutes \\
		
		\centerline{
			\xymatrix@R=0.5cm{
				A\ar[r]^-{\cong}\ar[d]_{h} & A^\circ \ar @{^{(}->}[r]^-{=} \ar @{-->}[d]^{h^\sharp}& \Gamma(A^\sharp,u)\ar @{^{(}->}[r]\ar @{-->}[d]^{\Gamma(h^\sharp)} & A^\sharp\ar @{-->}[d]^{h^\sharp}\ar @{^{(}->}[r]& \prod\limits_{P\in Spec\, A}(A/P)^\sharp\ar[d]^{\mathbf{h}^\sharp} \\  
				B\ar[r]^-{\cong}& B^\circ\ar @{^{(}->}[r]^-{=}  &\Gamma(B^\sharp,u )\ar  @{^{(}->}[r]& B^\sharp \ar @{^{(}->}[r] & \prod\limits_{Q\in Spec\, B}(B/Q)^\sharp \\
		}}
		According to theorem 3.3 on \cite{Yuri3}, $h^\sharp$ is a homomorphism of $l_u$-groups  and $\mathbf{h}^\sharp$ is a homorphims of $l$-groups. Recall that on the proof of theorem 3.3 on \cite{Yuri3}, for any $Q\in Spec(B)$ the well defined morphism $h|_Q \colon A/h^{-1}Q \to B/Q$; $h|_Q\left([a]_{h^{-1}Q}\right)= [h(a)]_Q$, makes the following diagram commute 
		
		\centerline{
			\xymatrix@R=0.5cm@C=0.5cm{
				A/h^{-1}Q\ar[r]^-{i}\ar[d]_{h|Q} & (A/h^{-1}Q)^\sharp\ar[d]^{(h|Q)^\sharp} \\  
				B/Q\ar[r]^-{i}&(B/Q)^\sharp \\
		}}
		and therefore the group homomorphims $\mathbf{h}^\sharp$ can be defined as follows:\\ given $\sigma \in \prod\limits_{P\in Spec\, A}(A/P)^\sharp,$ $\mathbf{h}^\sharp(\sigma)(Q)= (h|_Q)^\sharp(\sigma(h^{-1}Q))$, $(\mathbf{h_1h_2})^\sharp= \mathbf{h_1}^\sharp\mathbf{h_2}^\sharp$, and  $\mathbf{h}^\sharp|_{A^\sharp}= h^\sharp$.
		
		To finish the proof, it is enough to show that  $h^\sharp$ respects products in the gene\-rators of $A^\circ.$ \\
		Given $P\in Spec\,A,$ it follows from proposition \ref{numfun} that \\ $h^\sharp\left[\left(0,\widehat{a}\right)\left(0,\widehat{b}\right)\right](P)=h^\sharp\left[ (0,[a]_P) (0,[b]_P) \right]
		=h^\sharp\left[ (0,[a]_P)\right] h^\sharp\left[(0,[b]_P) \right]$ .

		As seen in the affirmation \ref{anum}, $A^\sharp$ is a semi-low $l_u$-ring, and if $h \colon A \to B$ is a homomorphism of $PMV_f$-algebras, $h^\sharp \colon A^\sharp \to B^\sharp$ is a homomorphism of semi-low $l_u$-rings. 
		
		On the other hand, given $\xymatrix@C=12pt{A \ar[r]^-h & B \ar[r]^-g & C} \in \mathcal{PMV}_f$, $(gh)^\sharp = g^\sharp h^\sharp$ follows directly from definition \ref{hsharp}.
	\end{proof}
	\begin{Teo}
		$\Gamma \colon \mathcal{LR}_u \to \mathcal{PMV}_f$ is a functor, where $\Gamma (R,u)= [0,u]$ and $\Gamma(h)= h|_{[0,u]}$ for every homomorphism of $l_u$-rings $h \colon R \to R'$.
	\end{Teo}
	\begin{proof}
		It is follows directly from corollary \ref{GammaPMV} and the affirmation \ref{Gamma-funt}.
	\end{proof}
	
	\subsubsection{The equivalence}
	\begin{Teo} Given any $PMV_f$-algebra  $A$ and semi-low $l_u$-ring $(R,u),$ the fo\-llowing homomorphisms are isomorphisms of $PMV_f$-algebras and semi-low $l_u$-rings. 
		\begin{center}
			$A\cong \Gamma(A^\sharp, u)$ and $(R,u)\cong (\Gamma(R,u))^\sharp.$
		\end{center}
	\end{Teo}
	\begin{proof} 
		For the first isomorphism $A\cong A^\circ$ as $PMV_f$-algebras, with $A^\circ\subset A^\sharp$ as shown in the following commutative diagram 
		
		\centerline{
			\xymatrix@R=0.5cm@C=1.2cm{
				A\,\, \ar@{>->}[r]^-{\widehat{(-)}_A}& \prod\limits_{P\in Spec\,A} (A/P)\,\,\ar@{>->}[r]^-{i}&\prod\limits_{P\in Spec\,A} (A/P)^\sharp\\
				&\ar@{<-<}[ul]^{\cong}A^\circ\subset A^\sharp \ar@{^{(}->}[ur] &\\ }}
		
		where $i$ is built using the universal property as follows: \\
		\centerline{
			\xymatrix@R=0.7cm@C=0.5cm{
				A/P\ar[r]^{i_P} & (A/P)^\sharp\\
				\prod\limits_{P\in Spec\,A} A/P\ar[u]^{\pi_P} \ar@{-->}[r]^-{\exists\,!i}& \prod\limits_{P\in Spec\,A} (A/P)^\sharp \ar[u]_{\pi^\sharp_P}
		}}
		with $i_P$ the application defined for each $P$ as  
		$$\begin{array}{rccl}
		i_P:& A/P&\longrightarrow &(A/P)^\sharp \\
		&[a]_P&\longmapsto &(0, [a]_P).\\
		\end{array}$$
		Because of  theorem \ref{generado} \textit{a}), $A^\circ=\Gamma(A^\sharp,u),$ and so $A\cong \Gamma(A^\sharp, u).$
	\end{proof}
	
	On the other hand, the isomorphisms of the chain semi-low $l_u$-rings,  obtained from the Chang's construction in \ref{ecuchang}, in the  theorem \ref{Equivalencia-cadenas} on the fibers of $\Gamma(R,u)^\sharp$ determine an  isomorphism of semi-low $l_u$-rings $\tau_{_R}:\Gamma(R,u)^\sharp\longrightarrow (R,u)$, as follows:\\
	
	\centerline{
		\xymatrix@R=0.5cm@C=1.2cm{
			\Gamma(R,u)^\sharp\ar@{^{(}->}[r]_-{id}^-{\subset} \ar@{-->}[d]_{\tau_{_R}}& \prod\limits_{P\in Spec\,R}(\Gamma(R,u)/P\cap[0,u])^\sharp\,\,\ar@{>->}[r]^-{\Theta}_-{\cong}&\prod\limits_{P\in Spec\,R} \Gamma(R/P,u_P)^\sharp\ar@{->}[d]^{v}_{\cong}\\
			(R,u)\,\, \ar@{>->}[rr]^-{\widehat{(-)}^g} & &\prod\limits_{P\in Spec\,R} (R/P,u_P)                               
	}}
	where $\tau_{_R}(0,\hat{x}) = \left[\widehat{(-)}^g\right]^{-1}(v\Theta(0,\hat{x}))=\left[\widehat{(-)}^g\right]^{-1}(\hat{x}^g). $ 
	
	$\tau_{_R}$ is well defined because the homomorphism $\widehat{(-)}^g$ is injective. The fact that $\tau_{_R}$ is injective follows from the fact that for every $x,y \in \Gamma(R,u)$,  $\widehat{x}^g = \widehat{y}^g$ implies $x=y$, and so $\widehat{x}=\widehat{y}$, since $\widehat{(-)}$ is a homomorphism. $\tau_{_R}$ is surjective because for every $x \in (R,u)$ it holds that $x = \sum_{i=1}^{n} \epsilon_{i}x_i$ for some  $x_i \in [0,u]$ by the theorem \ref{generado}, \textit{b}). Consequently $\tau_{_R}( \sum_{i=1}^{n}\epsilon_i(0, \widehat{x_i}))= \sum_{i=1}^{n}\epsilon_i\tau_{_R}(0,\widehat{x_i})=x$.
	\begin{Teo}
		For every $A \in \mathcal{PMV}_f$ and $R \in \mathcal{LR}_u$ the isomorphisms 
		
		\centerline{
			\xymatrix@R=2pt{A \ar[r]^-{i\widehat{(-)}_A} & \Gamma(A^\sharp,u) && \Gamma(R,u)^\sharp \ar[r]^-{\tau_{R}} & (R,u)    }}
		are natural transformations. 
	\end{Teo}
	\begin{proof}
		It follows directly from theorem 3.3 of \cite{Yuri3}.
	\end{proof}
	
	\section{$PMV_f$ vs $f$-rings}\label{pmv-anil}
	
	\begin{Def}($f$-rings [\cite{Birkhoff},XVII.5]). A function ring or $f$-ring is an $l$-ring that satisfies 
		\begin{center}	$a\wedge b=0$ and $c\geq 0$ implies $ac\wedge b=a\wedge cb=0.$
		\end{center}
	\end{Def}
	
	\begin{Pro}\label{prop-f-anillo} [\cite{Birkhoff},XVII.5].	In every $f$-ring it holds that 
		
		$$a\wedge b=0\implies ab=0.$$	
		
	\end{Pro}	
	
	\begin{Teo}\label{Fuchs} (Fuchs [\cite{Birkhoff},XVII.5]). An  $l$-ring is an $f$-ring if and only if  all its $l$-closed ideals are $L$-ideals.
	\end{Teo}
	\begin{Pro}
		For any $PMV_f$ $A,$ $A^\sharp$ is a semi-low $f_u$-ring.
	\end{Pro}
	\begin{proof}
		From affirmation \ref{anum}, $A^\sharp$ is a semi-low $l_u$-ring. From corollary \ref{l-L} and theorem \ref{Fuchs}, it is an $f$-ring, since all its $l$-ideals are $L$-ideals.
	\end{proof}
	\begin{Exp}
		$[0,1]^\sharp = \R.$
	\end{Exp}
	
	\begin{Pro}\label{dinoale}
		$PMV \subset MVW$-rig.
	\end{Pro}
	\begin{proof}
		From theorem 4.2 of \cite{DiNola2}, it follows that  for any $PMV$-algebra $A$ there exists an $l_u$-ring $R$ such that $\Gamma(R,u) \cong A$ and because of proposition \ref{Gamma}, $A$ is an $MVW$-rig. The inclusion is strict because of remark \ref{nopmv}. 
	\end{proof}
	
	\begin{Afir}
		For any set  $X$ the semi-low $f_u$-ring associated to the $PMV_f$-boolean algebra $2^X$ is isomorphic to the ring of bounded functions  of $\Z^X$, $B(\Z^X)$. 
	\end{Afir}
	\begin{proof}
		It is enough to see that $\left(2^X\right)^\sharp \cong B(\Z^X)$. The application $\Theta$ defined on the  generators, for all $f \in 2^X$, 
		$$\xymatrix@R=0.05cm{\left(2^X\right)^\sharp \ar[r]^-\Theta &  B(\Z^X) &\\
			i\widehat{f} \ar@{|->}[r] & \tilde{f} \colon X \ar[r] & \Z \\
			& \hspace{12pt} x \ar@{|->}[r] & f(x) }$$ with $\Theta\left(\sum_{j=1}^{k}i\widehat{f_j}\right) = \sum_{j=1}^{k}\Theta(i\widehat{f_j})$ \,\, and \,\,  $\Theta\left( (i\widehat{f})(i\widehat{g})\right) = \Theta( (i\widehat{f})) \Theta\left((i\widehat{g})\right),$ is a ring  isomorphism.  
		
		Because $2^X$ is a  hyper-archimedean $MV$-algebra, every prime ideal is maximal and of the form $P_x=\{f \in 2^X \colon f(x)=0 \}$, for all $x\in X$ and $[f]_{P_x} = f(x)$. Then, if $x \in X$, $\tilde{f}(x) \neq \tilde{g}(x) \Leftrightarrow f(x) \neq g(x) \Leftrightarrow f \neq g \Leftrightarrow [f]_{P_x} \neq [g]_{P_x} \Leftrightarrow \widehat{f} \neq \widehat{g}$, implies that $\Theta$ is well defined and injective, with   $P_x \in Spec(2^X)$.
		
		On the other hand, for $h \in B(\Z^X)$ it holds that $$h = \sum_{k=-n}^{n}k\lambda_k$$ with $|h|\leq n, \,\, \lambda_k \in 2^X$ such that $\lambda_k(x)= 1$ if $h(x)=k$ and zero elsewhere. Therefore $\Theta$ is surjective. By construction  $\Theta$ is a homomorphism of $l$-rings.
	\end{proof}
	\begin{Exp}
		The semi-low $f_u$-ring 	$(2^n)^\sharp$ is isomorphic to the ring  $\Z^n$.
	\end{Exp}
	\begin{Cor}
		Every boolean algebra seen as a $PMV_f$-algebra is  a subalgebra of $2^X$ for some set $X$. Since the functor $(-)^\sharp$ preserves subalgebras, the semi-low $f_u$-ring associated to a boolean algebra is a subring of the semi-low $f_u$-ring $B(\Z^X)$.
	\end{Cor}
	
	\begin{Exp}
		$F[x] \subset \mathcal{C}\left([0,1]^{[0,1]}\right),$ the semi-low $f_u$-ring of continuous functions defined as follows:
		$$f \in F[x] \Leftrightarrow \exists P_1,\cdots, P_k \in \Z[x], \,\, \text{such that}\,\, \forall x \in [0,1] \, f(x)= P_i(x),$$ for some $1\leq i \leq k,$ is the semi-low $f_u$-ring associated to the  $PMV_f$-algebra $\Gamma(F[x])$. This algebra is the  minimum $PMV_f$-algebra that contain the $MV$-algebra $Free_1$.
	\end{Exp}

	\section{The category $\mathcal{PMV}_f$ is coextensive}\label{co-extensividad}
	A category $\mathcal{C}$ is coextensive if only if $\mathcal{C}^{op}$ is extensive. 
	
	\begin{Def}\cite{Car} A category with finite products is coextensive if only if the projections of product is the terminal object and for all $g \colon A\times B \to C$, the following pushout exists and $C\cong C_1\times C_2.$\\
		
		\centerline{
			\xymatrix@R=0.5cm@C=0.5cm{
				A\ar @{<-}[r]^-{\pi_A} \ar @{-->}[d]& A\times B\ar @{->}[r]^-{\pi_B}\ar [d]^g & B\ar @{-->}[d]\\  
				C_1\ar @{<-}[r]  &C\ar  @{->}[r]& C_2  \\
		}}

	\end{Def}
	\begin{Obs}
		The terminal object of $\mathcal{PMV}_f$ category, is the $PMV_f$-algebra  $\{0\} = \mathbf{1}$.          
	\end{Obs}

	\begin{Pro}
		The category $\mathcal{PMV}_f$ is  coextensive.
	\end{Pro}
	
	\begin{proof}
		The pushout of projections  $\pi_A$ and $\pi_B$ of $PMV_f$-algebras $A,B$, is the terminal object because of for all $\lambda_A, \lambda_B$, $\lambda_A \pi_A = \lambda_B\pi_B$ implied $\lambda_A=\lambda_B=0$. It is enough to see that $(0,1)\in A\times B$ implies $\lambda_A\pi_A(0,1)=0=\lambda_B\pi_B(0,1)=1.$ 
		
		\centerline{
			\xymatrix@R=0.5cm@C=0.5cm{
				A\times B\ar @{->}[r]^-{\pi_B}\ar @{->}[d]_{\pi_A} & B\ar  @{->}[d]\ar @/^/[ddr]^{\lambda_B}\\  
				A\ar  @{->}[r]\ar  @/_/[drr]_{\lambda_A} & \mathbf{1} \ar@{.>}[dr]|-{\lambda}\\
				& & C 
		}}
		
		Let $g:A\times B\rightarrow C,$ an homomorphism the $PMV_f$-algebras, we named $g(0,1)=e,$ to idempotent element of $C$, and thus $g(1,0)=\neg e$ is idempotent too. We show that $\theta: C\rightarrow C/\left\langle e \right\rangle \times C/\left\langle\neg e\right\rangle; c\mapsto([c]_{\left\langle e\right\rangle},[c]_{\left\langle \neg e\right\rangle}),$ is an isomorphism of $PMV_f$-algebras, with $\left\langle e\right\rangle$ and $\left\langle \neg e\right\rangle$, the generated ideals of the subjacent $MV$-algebra of $C$. These ideals are ideals of the  $PMV_f$-algebra $C$, proposition \ref{espectros}.
		
		$\theta $ is well defined and is an homomorphism of $PMV_f$-algebras, proposition \ref{A/I}. It is injective because of exists  $c\in C$ such that, $\theta(c)=([c]_{\left\langle e\right\rangle},[c]_{\left\langle \neg e\right\rangle})=([0]_{\left\langle e\right\rangle},[0]_{\left\langle \neg e\right\rangle}),$ then $c\in \left\langle e\right\rangle$ and $c\in \left\langle \neg e\right\rangle,$ thus, $c\leq e$ and $c\leq \neg e$ because of $e$ and $\neg e$ are idempotent  elements, thus $c\leq e\wedge \neg e= g[(0,1)\wedge (1,0)] = 0.$

		The other hand, for all $x,y\in C,\, x\equiv y\,\, \text{mod} \left\langle\left\langle e \right\rangle,\left\langle \neg e \right\rangle\right\rangle,$ because of $\left\langle\left\langle e\right\rangle,\left\langle \neg e\right\rangle\right\rangle=C.$ From the Chinese Remainder Theorem, Lemma 2 \cite{Yuri2}, exists $c\in C$ such that $c\equiv x\,\, \text{mod} \left\langle e\right\rangle$ y $c\equiv y\,\, \text{mod} \left\langle \neg e\right\rangle,$ then $\theta(c)=([c]_{\left\langle e\right\rangle},[c]_{\left\langle \neg e\right\rangle})=([x]_{\left\langle e\right\rangle},[y]_{\left\langle \neg e\right\rangle}),$ for some $c\in C,$ thus $\theta$ is surjective.

		Now we defined $q_e= \pi_{ e} \theta$ con $\pi_e \colon   C/\left\langle e \right\rangle \times C/\left\langle\neg e\right\rangle \rightarrow C/\left\langle e \right\rangle $ and $q_A$ as follows: for all $a\in A$, $q_{A}(a)=q_e(g(a,b))$. $q_A$ is well defined because of $q_e(g(a,b))=q_e(g(a,b'))$. If $\theta(g(a,b))=([x]_{\left\langle e\right\rangle},[x]_{\left\langle \neg e\right\rangle})$ and $\theta(g(a,b'))=([y]_{\left\langle e\right\rangle},[y]_{\left\langle \neg e\right\rangle})$, it is enough to see that $[x]_{\left\langle e\right\rangle}=[y]_{\left\langle e\right\rangle}$.
		
		$\theta(g(0,b))=([z]_{\left\langle e\right\rangle},[z]_{\left\langle \neg e\right\rangle})$ with  $z\leq e.$ In effect, $([e]_{\left\langle e\right\rangle},[e]_{\left\langle \neg e\right\rangle})=\theta (g(0,1))=\theta(g(0,b\oplus \neg b))=\theta(g(0,b))\oplus \theta(g(0,\neg b))= ([z]_{\left\langle e\right\rangle},[z]_{\left\langle \neg e\right\rangle})\oplus ([w]_{\left\langle e\right\rangle},[w]_{\left\langle \neg e\right\rangle})=([z\oplus w]_{\left\langle e\right\rangle},[z\oplus w]_{\left\langle \neg e\right\rangle}),$ with $z,w \in C$. Thus, $z\leq z\oplus w\leq e.$
		
		Besides, $\theta(g(a,b))\ominus \theta(g(a,b'))=\theta(g(0,b\ominus b'))=([x\ominus y]_{\left\langle e\right\rangle},[x\ominus y]_{\left\langle \neg e\right\rangle}),$ and for previous affirmation, $x\ominus y\leq e,$ and $[x]_{\left\langle e\right\rangle}=[y]_{\left\langle  e\right\rangle}=q_A(a).$

		$q_A$ is an homomorphism of $PMV_f$-algebras by construction.
		
		Finally we show that the following diagrams are pushout. 
		
		\centerline{
			\xymatrix@R=0.5cm@C=0.5cm{
				A\ar @{<-}[r]^-{\pi_A} \ar @{-->}[d]^{q_A}& A\times B\ar @{->}[r]^-{\pi_B}\ar @{->}[d]^g & B\ar @{-->}[d]^{q_B}\\  
				C/\left\langle e\right\rangle\ar @{<-}[r]^-{q_e}  &C\ar  @{->}[r]^{q_{_{\neg e}}}& C/\left\langle \neg e\right\rangle
		}}
		
		Let $\lambda_A$ and $\lambda_g$ such that $\lambda_A \pi_A = \lambda_gg$, exists a unique $\lambda$ such that the following diagram is commutative, 
		
		\centerline{
			\xymatrix@R=0.5cm@C=0.5cm{
				& A\ar @{<-}[r]^-{\pi_A} \ar @{->}[d]^{q_A} \ar @/_/[dld]_{\lambda_A}& A\times B \ar @{->}[d]^g \\  
				&C/\left\langle e\right\rangle\ar @{<-}[r]^-{q_e} \ar @{.>}[dl]|-{\lambda} &C\ar @/^/[dll]^{\lambda_g} \\
				P	& &
		}} 
		

		We defined $\lambda([c]_{\left\langle e\right\rangle})=\lambda_{g}(c).$ $\lambda $ is well defined, because of $[c]_{\left\langle e\right\rangle}=[c']_{\left\langle e\right\rangle}\leftrightarrow c\ominus c'\leq e,$ and $\lambda_g(e)=\lambda_g(g(0,1))=\lambda_{A}\pi_{A}((0,1))=0.$ $\lambda q_e = \lambda_g$ by construction, and $\lambda q_A(a) = \lambda q_e(g(a,b)) = \lambda_g g(a,b) = \lambda_A \pi_A(a,b)= \lambda_A(a)$, with $b \in B.$	$\lambda$ is unique by construction.
		
		The similar form we show that $q_{\neg e }g = q_B \pi_B$, is a pushout.		
	\end{proof}
	
	\begin{Cor} The category $\mathcal{PMV}_1$ defined by  Montagna \cite{Montagna}, is coextensive.
	\end{Cor}
	
	\section{Conclusions} 
	The construction of Dubuc-Poveda \cite{Yuri3} lets you visualize the associate ring of each $PMV_f$-algebra, because this do not use the \textit{good sequences}, and used the easy construction by Chang \cite{Chang2} for chains. The explicit construction of this equivalence permit us to study some properties of  commutative algebra for the class of semi-low $f_u$-rings, in relationship with $PMV_f$-algebras. We know about the problem to study the free algebras of  $f_u$-rings, however, its relationship with the $PMV_f$-algebras  will let to see this from the other perspective.

\end{document}